\documentclass[11pt]{amsart}

\usepackage[top=1.5in, bottom=1in, left=0.9in, right=0.9in]{geometry}

\usepackage{amscd,amsmath,amssymb,fancyhdr,color}
\usepackage[utf8]{inputenc}
\usepackage{amsfonts}
\usepackage{amsthm}
\usepackage{accents}
\usepackage{graphicx}
\usepackage{float}
\usepackage{subcaption}
\usepackage{verbatim}
\usepackage{indentfirst}
\usepackage{dsfont}
\usepackage{tikz}
\usepackage{tikz-cd}
\usetikzlibrary{matrix}
\usepackage[all]{xy}
\usepackage{enumerate}
\usepackage{extarrows}
\usepackage{csquotes}

\usepackage{scalerel,stackengine}
\stackMath
\newcommand\reallywidehat[1]{%
	\savestack{\tmpbox}{\stretchto{%
			\scaleto{%
				\scalerel*[\widthof{\ensuremath{#1}}]{\kern-.6pt\bigwedge\kern-.6pt}%
				{\rule[-\textheight/2]{1ex}{\textheight}}%WIDTH-LIMITED BIG WEDGE
			}{\textheight}% 
		}{0.5ex}}%
	\stackon[1pt]{#1}{\tmpbox}%
}

\usepackage{faktor}

\usepackage{BOONDOX-uprscr}

\usepackage[backref=page]{hyperref}
\renewcommand*{\backref}[1]{}
\renewcommand*{\backrefalt}[4]{%
	\ifcase #1 (Not cited.)%
	\or        (Cited on page~#2.)%
	\else      (Cited on pages~#2.)%
	\fi}

\hypersetup{
	colorlinks   = true,
	citecolor    = magenta
}

\newcommand{\Lg}{\mathfrak{g}}
\newcommand{\La}{\mathfrak{a}}
\newcommand{\Ln}{\mathfrak{n}}
\newcommand{\Lr}{\mathfrak{r}}
\newcommand{\Lt}{\mathfrak{t}}
\newcommand{\Lh}{\mathfrak{h}}

\numberwithin{equation}{section}

\def\eqref#1{(\ref{#1})}

\newcommand{\N}{{\mathbb N}}
\newcommand{\Z}{{\mathbb Z}}
\newcommand{\C}{{\mathbb C}}
\newcommand{\R}{{\mathbb R}}
\newcommand{\CP}{{\mathbb{CP}}}

\newcommand{\del}{\partial}

\def\1{\sqrt{-1}\:}

\newcommand{\cntrct}                % contraction with a vector field
{\hspace{2pt}\raisebox{1pt}{\text{$\lrcorner$}}\hspace{2pt}}

\newcommand{\End}{\operatorname{End}}

\newcommand{\id}{\operatorname{\text{\sf id}}}

\newcommand{\Aut}{\operatorname{Aut}}

\renewcommand{\dim}{\operatorname{dim}}

\renewcommand{\Re}{\operatorname{Re}}

\newcounter{Mycounter}[section]
\newcounter{lemma}[section]
\newcounter{claim}[section]
\newcounter{sublemma}[section]
\newcounter{corollary}[section]
\newcounter{theorem}[section]
\newcounter{conjecture}[section]
\newcounter{proposition}[section]
\newcounter{definition}[section]
\newcounter{example}[section]
\newcounter{remark}[section]
\newcounter{problem}[section]
\newcounter{question}[section]
\makeatletter

\@addtoreset{equation}{section}

\@addtoreset{footnote}{section}

\makeatother

\makeatletter
\DeclareRobustCommand*{\mfaktor}[3][]
{
   { \mathpalette{\mfaktor@impl@}{{#1}{#2}{#3}} }
}
\newcommand*{\mfaktor@impl@}[2]{\mfaktor@impl#1#2}
\newcommand*{\mfaktor@impl}[4]{
   \settoheight{\faktor@zaehlerhoehe}{\ensuremath{#1#2{#3}}}%
   \settoheight{\faktor@nennerhoehe}{\ensuremath{#1#2{#4}}}%
      \raisebox{-0.5\faktor@zaehlerhoehe}{\ensuremath{#1#2{#3}}}%
      \mkern-4mu\diagdown\mkern-5mu%
      \raisebox{0.5\faktor@nennerhoehe}{\ensuremath{#1#2{#4}}}%
}
\makeatother

\usetikzlibrary{arrows,chains,matrix,positioning,scopes}

\makeatletter
\tikzset{join/.code=\tikzset{after node path={%
			\ifx\tikzchainprevious\pgfutil@empty\else(\tikzchainprevious)%
			edge[every join]#1(\tikzchaincurrent)\fi}}}
\makeatother

\tikzset{>=stealth',every on chain/.append style={join},
	every join/.style={->}}

\makeatletter
\newtheorem*{rep@theorem}{\rep@title}
\newcommand{\newreptheorem}[2]{%
	\newenvironment{rep#1}[1]{%
		\def\rep@title{\ref{##1}}%
		\begin{rep@theorem}}%
		{\end{rep@theorem}}}
\makeatother

\newreptheorem{theorem}{Theorem}

\newtheoremstyle{named}{}{}{\itshape}{}{\bfseries}{.}{.5em}{\thmnote{#3's }#1}
\theoremstyle{named}

\begin{document}
	
	\newpage
	
	\title{On some properties of Hopf manifolds}

  \author{Nicolina Istrati}

\address{Nicolina Istrati \newline
		\textsc{\indent Département de mathématiques, Faculté des Sciences, Université d'Angers\newline 
			\indent 2 Boulevard Lavoisier
49000 Angers, France \newline}}
	\email{nicolina.istrati@univ-angers.fr}

\author{Alexandra Otiman}
	\address{Alexandra Otiman \newline
		\textsc{\indent Institut for Matematik, Aarhus University\newline 
			\indent 8000, Aarhus C, Denmark\newline
			\indent \indent and\newline
			\indent Institute of Mathematics ``Simion Stoilow'' of the Romanian Academy\newline 
			\indent 21 Calea Grivitei Street, 010702, Bucharest, Romania}}
	\email{aiotiman@math.au.dk}

\thanks{N.I. was partially supported by the PNRR-III-C9-2023-I8 grant CF 149/31.07.2023 "Conformal Aspects of Geometry and Dynamics" and by a CNRS PEPS JCJC project. A.O. is supported by a DFF Sapere Aude grant ``Conformal geometry: metrics and cohomology" and by a grant of the Ministry of Research and Innovation, CNCS-UEFISCDI, project no. P1-1.1-PD-2021-0474}
	%\date{\today}
\dedicatory{Dedicated to Professor Paul Gauduchon for his 80th birthday.}	
	
	\begin{abstract}
		We review old and new properties of Hopf manifolds from the point of view
of their analytic and metric structure. 
    % by focusing on analytic, metric and cohomological properties.
	\end{abstract}
	
	\maketitle
	
	\hypersetup{linkcolor=blue}
%	\tableofcontents

	\section{Introduction}

    Hopf manifolds constitute the main examples of compact complex manifolds not admitting any compatible Kähler metric. These manifolds date back to Hopf \cite{Hopf}, who first studied complex manifolds having the diffeomorphism type of $\mathbb{S}^1\times\mathbb{S}^{2n-1}$. Later on, motivated by the classification of compact complex surfaces, Kodaira \cite{k} reviewed his definition and proposed a more general one, which is the one we adopt here, namely Hopf manifolds are compact complex manifolds of dimension $n \geq 2$ whose universal cover is $\mathbb{C}^n \setminus \{0\}$. Equivalently, they are finite quotients of  $W_\gamma:=\mathbb{C}^n \setminus \{0\}/\langle\gamma\rangle$ for $\gamma\in\Aut(\C^n,0)$ a contraction.

    The geometry of Hopf manifolds is very rich and they have been intensely studied, particularly as model spaces bearing locally conformally Kähler metrics. The aim of this short survey is to discuss some of the more fundamental %old and new 
    properties of Hopf manifolds, concerning both the analytic invariants and the metric structure.

The paper is organized as follows. In Section \ref{sec: NormalForm}, we review the general theory of Hopf manifolds, their classification according to the Poincaré-Dulac normal form of the contracting
germ $\gamma$  and several properties concerning their analytic deformations.
This part mostly follows \cite{h}. In Section \ref{sec: Autom}, we discuss in more detail the structure of the automorphism group of a Hopf manifold $W_\gamma$. In particular, this part allows us to determine the Lee group of a Vaisman metric when $\gamma$ is diagonal, and to show that when $\gamma$ is not diagonal, no such Lee group, and therefore no Vaisman metric can exist (\ref{prop: HLee}). This result is possibly new, but generalizes Belgun's proof \cite{bel} in dimension $2$.

In Section \ref{sec: Coh} we revisit analytical and cohomological properties, such as the Dolbeault and the Bott-Chern cohomology and the Kodaira dimension of Hopf manifolds of the form $W_\gamma$. Moreover, for the case when $\gamma$ is diagonal, we determine the algebraic dimension of $W_\gamma$ in terms of the relations satisfied by the eigenvalues of $d_0\gamma$ (\ref{thm: AlgDim}). This generalizes \cite[Thm. 31]{k}, which determines the algebraic dimension of all Hopf surfaces.

Finally, Section \ref{sec: Metrics} is concerned with the metric properties of Hopf manifolds, with a focus on locally conformally Kähler (lcK) metrics.
Gauduchon and Ornea \cite{go} gave the first proof that Hopf surfaces admit lcK metrics, and in the diagonal case, even Vaisman ones. Since then, multiple generalizations have been given, see for instance  \cite{bel, ko,  bru, ov16}. In the last section, we briefly  recall the history of the problem and choose to present Brunella's proof \cite{bru} adapted to the Hopf setting, which is the most general one in some sense, as it does not distinguish among the special froms of the corresponding contraction. We end  with some non-existence results concerning other distinguished Hermitian metrics, following \cite{c, pop14}.

     \section{Normal forms of contractions and the classification of Hopf manifolds}  \label{sec: NormalForm}

   \begin{definition}
A \textit{Hopf manifold} of dimension $n\geq 2$ is a compact complex manifold with universal cover biholomorphic to $W:=\C^n\setminus\{0\}$. It is called \textit{primary} if it has infinite cyclic fundamental group, and \textit{secondary} otherwise.   
\end{definition}

Hopf's initial definition \cite{Hopf}, which is also sometimes the one used in the literature, is the one corresponding to primary Hopf manifolds. In fact, a general Hopf manifold is very close to a primary one, by the following result:

\begin{theorem}(see \cite{k} in dimension $2$, \cite[Thm 2.1, Cor 2.2]{ha} in general dimension)
Let $X$ be a Hopf manifold. Then $\pi_1(X)$ is a finite extension of an infinite cyclic subgroup generated by a contraction. In particular, $X$ admits a primary Hopf surface as a finite unramified normal cover. 
\end{theorem}

In \cite{Ka}, Kato moreover classifies the posibilities for the finite groups giving rise to secondary Hopf manifolds, as well as the diffeomorphism types of these. Hasegawa \cite{ha} partially generalises Kato's results to any dimension. 

In any case, an equivalent definition of Hopf manifolds is:

\begin{definition}
A primary Hopf manifold of dimension $n\geq 2$ is a complex manifold  $W_\gamma$ obtained as the quotient of $W:=\C^n\setminus\{0\}$ by the $\Z$-action generated by a holomorphic contraction $\gamma\in\Aut(\C^n\setminus\{0\}).$ A secondary Hopf manifold is a finite quotient of a primary Hopf manifold. 
\end{definition}

Recall that  a \textit{contraction} $\gamma$ of $W$ is a biholomorphism of $W$ such that $\gamma^m$ converges uniformly to $0$ on compact neighbourhoods of $0$ when $m$ goes to infinity. Note that by Hartogs' lemma, any $\gamma\in\Aut(\C^n\setminus\{0\})$ extends naturally to an automorphism of $\C^n$ fixing $0$, and via Schwarz lemma it can be shown that $\gamma$ is a contraction precisely when the eigenvalues of $d_0\gamma$ are  all in the punctured unit disk $\mathbb{D}^*$ of $\C$. 

As explained in \cite[Appendix]{h} following \cite{a}, any contraction can be put into normal form after a holomorphic change of coordinates of $W.$ In order to explain this precisely, let us first introduce some terminology.

\begin{definition} 
Let $\lambda=(\lambda_1,\ldots,\lambda_n)\in(\C^*)^n$.
\textbf{A resonance relation} for $\lambda$ is a relation of the form 
\[ \lambda_s=\lambda^m:=\lambda_1^{m_1}\cdots\lambda_n^{m_n},\quad s\in\{1,\ldots,n\}, m=(m_1,\ldots,m_n)\in\N^n\setminus\{e_s\}.\]
We will also denote this by $(s;m)$.
Here $e_1,\ldots, e_n$ denotes the standard basis of $\C^n.$ If $\lambda$ has no resonance relation, we say that $\lambda$ is \textbf{non-resonant}.
\end{definition}

\begin{remark}\label{rem: finiteRel}
Note that if $\lambda\in(\mathbb{D}^*)^n$, then there is only a finite number of resonance relations for $\lambda.$ Indeed, without loss of generality we can assume that  $|\lambda_1|\leq|\lambda_2|\leq\ldots\leq|\lambda_n|<1$ and then note that for each $s\in\{1,\ldots n\}$, for a relation $\lambda_s=\lambda^m$ with  $m\in\N^n$, the length $|m|:=\sum_{k=1}^nm_k\in\N$ is bounded:
\[ \frac{\ln|\lambda_s|}{\ln|\lambda_1|}\leq|m|\leq \frac{\ln |\lambda_s|}{\ln|\lambda_n|}. \]
\end{remark}

\begin{definition}
Let $\lambda\in(\C^*)^n$. An analytic map $\gamma=(\gamma_1,\ldots,\gamma_n)\colon\C^n\rightarrow \C^n$ is called \textbf{$\lambda$-resonant} if each component $\gamma_i$ expresses as
\begin{equation}\label{eq: la-resonant}
 \gamma_i(z)=\alpha_iz_i+\sum_{m}a_{m,i}z^m\end{equation}
with $\alpha_i,a_{m,i}\in\C$ and the sum is taken over each $m\in\N^n$ corresponding to a resonance relation $(i;m)$ for $\lambda$.  We denote by $R_\lambda$ the subalgebra of $\lambda$-resonant analytic maps, and by $G_\lambda\subset R_\lambda$ the group of invertible elements.  
\end{definition}

\begin{remark}\label{rem: finRel}
According to \ref{rem: finiteRel}, if $\lambda\in(\mathbb{D}^*)^n$ then there are only finitely many resonance relations, therefore $R_\lambda$ consists only in polynomial maps and $G_\lambda$, being open in $R_\lambda$, is a finite dimensional connected complex Lie group. Its Lie algebra is given by
\[ \Lg_{\lambda}=\langle z_1\frac{\del}{\del z_1},\ldots z_n\frac{\del}{\del z_n}\rangle\oplus\langle z^m\frac{\del}{\del z_s}\mid (s;m)\rangle\]
where the second family is taken over all resonance relations $(s;m)$ for $\lambda$, and the sum is a vector space sum.
\end{remark}

\begin{remark}\label{rem: diagonal}
If we denote by 
\[ d_\lambda\colon\C^n\rightarrow\C^n,\quad z\mapsto (\lambda_1z_1,\ldots,\lambda_nz_n)\]
then it is easy to check that 
\[ R_\lambda=\{\gamma\colon \C^n\rightarrow\C^n\mid \gamma\circ d_\lambda=d_\lambda\circ\gamma\}. \]
Thus the above remark implies that $\Lg_\lambda=\mathfrak{aut}(W_{d_\lambda})$ and  $G_\lambda/\langle d_\lambda\rangle=\Aut_0(W_{d_\lambda})$.

\end{remark}

We have the following classification result:

\begin{theorem}\label{normalform}(Poincaré-Dulac normal form, cf. \cite[p. 187]{a} and \cite[Appendix]{h})\newline Let $\gamma\in\Aut(\C^n,0)$ be a contraction and let $\lambda=(\lambda_1,\ldots,\lambda_n)\in(\mathbb{D}^*)^n$ be the eigenvalues of $d_0\gamma$. Then there exists a biholomorphism $h\in\Aut(\C^n,0)$ and a contraction $\tilde\gamma\in G_\lambda$ so that $\tilde\gamma=h\gamma h^{-1}$. Moreover, one can pick $\tilde\gamma$ in upper triangular form, i.e. so that each component of $\tilde\gamma$ is of the form
\[\tilde\gamma_j(z)=\lambda_jz_j+g_j(z_{j+1},\ldots, z_{n}),\quad \alpha_j\in\C^*, j\in\{1,\ldots, n\}. \]
Such a $\tilde\gamma$ is called a normal form for $\gamma$.
\end{theorem}

\begin{definition}
A Hopf manifold of the form $W_{d_\lambda}$ for some $\lambda\in(\mathbb{D}^*)^n$ is called \textit{a diagonal Hopf manifold}. A primary Hopf manifold $W_\gamma$ for which $\gamma$ does not admit a diagonal normal form  is also called a \textit{a non-diagonal Hopf manifold.}
\end{definition}

 \subsection{ Moduli spaces and deformations}

    Note at this point that while $\lambda\in(\mathbb{D}^*)^n$ is an invariant of the Hopf manifold $W_{\gamma}$, the normal form $\tilde\gamma$ of $\gamma$ is not unique. Indeed, $G_\lambda$ acts on itself by conjugation and every element in the orbit $G_\lambda.\tilde\gamma$ gives a different representative of a normal form for the same Hopf manifold. In any case, as a consequence we obtain that $G_\lambda$ parametrizes all small deformations of a Hopf manifold $W_\gamma$, more specifically:

  \begin{theorem}\label{Thm: H}(see \cite{h} for the full statement, but also \cite{w}, \cite{b} for particular cases)
  
  Let $W_\gamma$ be a Hopf manifold with $\gamma\in G_\lambda$ in normal form and let $S\subset G_\lambda$ be a complex analytic slice passing through $\gamma$ and complementary to the orbit $G_\lambda.\gamma$. Then the family $(W_s)_{s\in S}$ forms a versal deformation of $W_\gamma.$ Moreover we have:
\begin{gather*}
    H^0(W_\gamma,TW_\gamma)\cong H^1(W_\gamma,TW_\gamma) \cong \ker(\id-\gamma_*\in\End(\Lg_\lambda)),\\
 H^k(W_\gamma,TW_\gamma)=0, \quad k>1.
\end{gather*}

\end{theorem}

For primary Hopf manifolds we have the following result, typically very useful in order to extend properties of diagonal Hopf manifolds to non-diagonal ones via upper-semicontinuity.

\begin{theorem}\label{deformation}(see for instance \cite[p. 242]{h})

Let $W_\gamma$ be a  Hopf manifold with $\gamma\in G_\lambda$ in normal form.%, so that each component of $\gamma$ has the form 
There exits a holomorphic family of Hopf manifolds $p\colon\mathcal{X}\rightarrow \C $ 
such that $p^{-1}(0)= W_{d_\lambda}$ and for any $t\in \C^*$ we have $p^{-1}(t)\cong W_\gamma$.
\end{theorem}

\begin{proof}
The proof of this result comes down to noticing that when the components of $\gamma$ are in the form \eqref{eq: la-resonant}, by conjugating with diagonal automorphisms, which live in $G_\lambda$, one can simultaneously make all the coefficients $a_{m,i}$ as small as one wishes.
%%%%%%%%
More specifically, let
\[\gamma=d_\lambda+(g_1,\ldots, g_{n}), \quad g_i(z)=\sum_ma_{m,i}z^m, \]
be in upper triangular normal form, so that for each coefficient  $a_{m, i}$ that does not vanish we have $m_j=0$ as soon as $1\leq j\leq i$, and in particular $n_{i,m}:=\sum_{j=1}^n jm_j-i>0$.

Define the family $\mathcal X:=W\times \C/\Z$ where $1\in\Z$ acts on $(z,t)\in W\times\C$ by
\[1.(z,t)=(d_{\lambda}(z)+g_t(z),t)\]
with the $i$-th component of $g_t$ being
\[\sum_{m}t^{n_{i,m}}a_{m,i}z^m.\]
The map $p$ is the naturally induced projection, so that $p^{-1}(0)=W_{d_\lambda}.$

For any $\mu=(\mu_1, \ldots, \mu_n)\in (\C^*)^n$ we have $d_\mu\in G_{\lambda}$ and the $i$-th component of $d_\mu\gamma d^{-1}_\mu$ is given by
\[ \lambda_iz_i+\sum_ma_{m,i}\frac{\mu_i}{\mu^m}z^m. \]
In particular, for $t\in\C^*$ and $\mu:=(t^{-1},t^{-2},\ldots, t^{-n})$ we obtain that $d_{\mu}\gamma d_{\mu}^{-1}=d_{\lambda}+g_t$ and so $W_\gamma\cong p^{-1}(t)$, which concludes the proof. \end{proof}

\section{Automorphisms}\label{sec: Autom}
    
Given $\lambda\in(\mathbb{D}^*)^n$ and $\gamma\in G_\lambda$ with eigenvalues $\lambda=(\lambda_1,\ldots,\lambda_n)$, let us denote by \[\Lg_\gamma:=\Lg_\lambda^{\langle\gamma_*\rangle}=H^0(W_\gamma, TW_\gamma)\]
where the last equality holds by \ref{Thm: H}. In particular, $\Aut_0(W_\gamma)=G_\gamma/\langle\gamma\rangle$, where $G_\gamma$ is the exponential of $\Lg_\gamma$ in the group $G_\lambda$.

Let us now make more explicit the Lie group structure of $G_\lambda$. Recall that by \ref{rem: finRel} and \ref{rem: diagonal}, the group $G_\lambda$ has Lie algebra 
\[ \Lg_{\lambda}=\langle z_1\frac{\del}{\del z_1},\ldots z_n\frac{\del}{\del z_n}\rangle\oplus\langle z^m\frac{\del}{\del z_s}\mid (s;m)\rangle=\La\oplus\Lr.\]
 Note that the first summand $\La$ corresponds to the Lie algebra of the standard torus $(\C^*)^n\subset\Aut(W)$. 
A vector field of the form $\xi_{s,m}=z^m\frac{\del}{\del z_s}\in\Lr$ generates the effective action of $\C\subset\Aut(W)$ given by
\[ \Phi^t_{\xi_{s,m}}(z)=z+tz^me_s,\quad t\in\C, z\in W.\]
Moreover, $\Lr$ has a natural finite $\N$-grading given by
\[\Lr_l:=\langle \xi_{s,m}|(s;m), l=|m|\rangle\quad l\in\N, l\geq 1.\]

Denoting by $Z_j=z_j\frac{\del}{\del z_j}=\xi_{j,e_j}$, $j\in\{1,\ldots,n\}$, we have the following relations:
\begin{equation}\label{eq: LieRel} [Z_j,\xi_{s,m}]=(m_j-\delta_{j,s})\xi_{s,m};\quad [\xi_{s,m},\xi_{t,n}]=n_s\xi_{t,m+n-e_s}-m_t\xi_{s,m+n-e_t}. \end{equation}
In particular, $[\Lr_a,\Lr_b]\subset \Lr_{a+b-1}$ every time $(a,b)\neq (1,1)$, so that $\Ln_{\lambda}:=\oplus_{l\geq 2}\Lr_l$ is a nilpotent ideal of $\Lg_\lambda$ and \[\Lg_\lambda^{red}:=\La\oplus\Lr_1\subset \mathfrak{gl}(n,\C)\]
is a sub-algebra of $\Lg_\lambda$.
Moreover, the equivalence relation $\sim$ on $\{1,\ldots, n\}$ defined by $i\sim j$ $\Leftrightarrow$ $\lambda_i=\lambda_j$ gives us a partition $\{1,\ldots, n\}=\bigsqcup_{i=1}^kI_i$ so that for each orbit $I_i$, denoting by $l_i$ its cardinal, we readily find, via \eqref{eq: LieRel}, isomorphisms of Lie algebras
\[\Lg_i:=\langle \xi_{j,e_k}| j,k\in I_i\rangle\cong \mathfrak{gl}(l_i,\C),\quad \Lg_\lambda^{red}\cong \bigoplus_{i=1}^k\Lg_i.\]
In particular, $\Lg_\lambda^{red}=\Lg_\lambda/\Ln_{\lambda}$ is reductive, implying that $\Ln_{\lambda}$ is the nilradical of $\Lg_\lambda$ and we have a semi-direct decomposition 
\begin{equation}\label{eq: Gred}
\Lg_\lambda=\Ln_{\lambda}\rtimes \Lg_\lambda^{red}.\end{equation}
Denoting by $N:=\mathrm{exp}(\Ln_{\lambda})\subset\Aut(W)$ and by $G_{\lambda}^{red}:=\exp(\Lg_\lambda^{red})\subset\mathrm{GL}(n,\C)\subset\Aut(W)$, this lifts to %$G_\lambda=N\rtimes G_\lambda^{red}$.
\begin{equation}\label{eq: GroupRed}
G_\lambda=N\rtimes G_\lambda^{red}, \quad G_\lambda^{red}=\prod_{i=1}^k\mathrm{GL}(l_i,\C) \end{equation}
where the product group is identified with a subgroup of  $\mathrm{GL}(n,\C)$  via the standard block-diagonal embedding.

Note that the group $N$ has no compact subgroup, since any vector field of $\Ln_{\lambda}$
%\[\xi=\sum_{j,m}a_{m,j}\xi_{j,m}\in\Ln_{\lambda}\]
generates an effective $\C$-action. In particular,  the exponential map $\Ln_\lambda\rightarrow N$ is a biholomorphism and $N$ is simply connected.

\begin{remark}
If $\lambda$ has only resonance relations $(s;m)$ of degree $|m|\geq 2$, i.e. if $\Ln_{\lambda}=\Lr$,  then $\Ln_{\lambda}=[\Lg_\lambda,\Lg_\lambda]$ and  $\Lg_\lambda$ has the structure of a complex \textit{splittable} solvable Lie algebra, i.e.
 $\Lg_\lambda=\La\ltimes_\phi \mathfrak{n}$ for  $\phi:\La\rightarrow \mathfrak{aut}(\Ln_{\lambda})$ given by the adjoint representation. In particular, we have an exact sequence of complex Lie groups:
\begin{equation*}
   \begin{tikzcd}
   0\rar &N\rar &G_\lambda\rar &(\C^*)^n\rar&0.
   \end{tikzcd}
\end{equation*}
\end{remark}

 \begin{proposition}\label{prop: HLee}
 Let $W_\gamma$ be a primary Hopf manifold. If $W_\gamma$ is non-diagonal, there exists no compact torus $H\subset\Aut_0(W_\gamma)$ so that $J\Lh\cap\Lh\neq \{0\}.$
If $\gamma=d_\lambda$, then the only complex one-dimensional Lie group $G$ sitting in a compact torus $H\subset\Aut_0(W_\gamma)$ is generated by 
\[Z:=\sum_{j=1}^n\ln |\lambda_j|z_j\frac{\del}{\del z_j}.\]
Its action on $W_\gamma$ is induced by the action of $\C^*$ via
 \[ \mu.[z_1,\ldots, z_n]=[\mu^{\ln |\lambda_1|}z_1,\ldots, \mu^{\ln |\lambda_n|}z_n],\quad \mu\in\C^*, [z]\in W_\gamma.\]

 \end{proposition}
 \begin{proof}
  Let $\gamma\in G_\lambda$ be in normal form, let us denote by $G:=G_\gamma/\langle \gamma\rangle\subset G_\lambda/\langle\gamma\rangle$, let 
 \[q\colon G_\lambda\rightarrow G_\lambda^{red},\quad q'\colon G_\lambda/\langle\gamma\rangle\rightarrow G_\lambda^{red}/\langle q(\gamma)\rangle\] 
 be the natural projections and we denote also by $G':=G_\lambda^{red}/\langle q(\gamma)\rangle$. Note that since $N\cap\langle\gamma\rangle=\{\id\}$, we have $N=\ker q'.$
 
  Let $H\subset G$ be a torus as in the statement. By eventually replacing it with a smaller torus, we can assume that $H_c:=\exp(J\Lh\cap\Lh)\subset H$ is dense in $H$. Since the group $N$ has no compact subgroup, $H\cap N$ is trivial and therefore $H$ projects isomorphically via $q'$ to a torus $H'$ in $G'$ whose Lie algebra $\Lh'$ also satisfies $J\Lh'\cap\Lh'\neq\{0\}$.
  
  By eq. \eqref{eq: GroupRed}, the exponential map of $G_\lambda^{red}$ is surjective, therefore we can write $\gamma':=q(\gamma)=\exp(\xi_\gamma')$ for some vector field $\xi_\gamma'\in\Lg_\lambda^{red}$. Let us denote by $\Lt\subset\La$ the Lie algebra of the real compact subtorus of $(\C^*)^n\subset G_\lambda$. Since any maximal torus in $G_\lambda^{red}$ is conjugated to $\exp(\Lt)$, we find that any maximal torus in $G'$ up to conjugation has Lie algebra either $\Lt$ or $\Lt\oplus\R\xi'_\gamma.$ Since such a torus must moreover contain the non-purely real Lie  group $H'$, we must be in the second case, so that $\Lh'\subset\Lt\oplus\R\xi'_\gamma.$ 
  
  Note however that if $\xi'_\gamma=\xi_a+\xi_\Lr$ with $\xi_\La\in \La, \xi_\Lr\in\Lr$, then the condition $J\Lh'\cap\Lh'\neq\{0\}$ implies that $\xi_\Lr=0$, since else we would have that $J\xi_r$ and $\xi_r$ span the same real line. This implies thus that $\gamma'\in (\C^*)^n$, and so $\gamma'=d_\lambda$ and $\gamma=\gamma_N\circ d_\lambda$ with $\gamma_N\in N.$ Note that since $\gamma$ is central in $G_\gamma$, $\gamma_N$ is central in $N_\gamma:=N\cap G_\gamma$.
  
Consider now the holomorphic map $\phi\colon G_\gamma\rightarrow N_\gamma$, $g\mapsto q(g)^{-1}g$ and the induced holomorphic map $\phi'\colon G \rightarrow N_\gamma/\langle\gamma_N\rangle $. Then  $\phi'|_{H_c}\colon H_c\rightarrow N_\gamma/\langle\gamma_N\rangle$ is holomorphic and bounded, since $\phi'(H_c)\subset \phi'(H).$ On the other hand, as $\gamma_N\in\mathcal {Z}(N_\gamma)$, we have that if $\gamma_N\neq\id$ then $N_\gamma/\langle\gamma_N\rangle$ is biholomorphic to $\C^*\times\C^{\dim N_\gamma-1}$ via the exponential map, therefore we find by Liouville's theorem that $\phi'|_{H_c}$ is constant, and hence also $\phi'|_H$ is constant by density. In particular, as $\hat H:=\exp\Lh\subset G_\gamma$ is connected, we also have that $\phi|_{\hat H}$ is constant and so $\hat H\subset G_\lambda^{red}.$ At the same time, the argument from before shows that $\hat H$ intersects $\langle \gamma\rangle$ non trivially, therefore $\gamma=d_\lambda\in G^{red}_{\lambda}$ is indeed diagonal.

When $\gamma=d_\lambda$ is diagonal, writing  $\lambda_j=r_je^{i\theta_j}$ for $j\in\{1,\ldots, n\}$, we find that $\Lh\cap J\Lh$ is generated over $\C$ by %the Lee group is generated by 
\[Z:=\sum_{j=1}^n\ln r_jz_j\frac{\del}{\del z_j}.\]
This concludes the proof. \end{proof}

     \section{Cohomology and analytic invariants}\label{sec: Coh}

We present in this section cohomological properties of primary Hopf manifolds. Since these have the tolological type of $\mathbb{S}^1\times\mathbb{S}^{2n-1},$ de Rham cohomology is readily computable. On the other hand, the complex cohomologies, such as Dolbeault and Bott-Chern, are more intricate and can be detected via the analytic family constructed by \ref{deformation}.

For the next result we follow the ideas of proof of Mall \cite{m} for which we provide a few more details. Note that different proofs are possible, by using for instance the existence of Vaisman metrics on diagonal Hopf manifolds.

\begin{theorem}{\cite[Thm. 3]{m}}\label{dolbeault}
The Hodge numbers of an $n$-dimensional primary Hopf manifold $W_{\gamma}$ are given by
\begin{equation*}
         h_{\overline{\partial}}^{p,q}=\begin{cases}
      1, & \text{if}\,\, (p, q) \in \{(0, 0), (0,1), (n-1,n), (n,n)\}\\
      0, & \text{otherwise.}
    \end{cases}      
      \end{equation*} 
      In particular, Hopf manifolds do not satisfy the $\partial\overline{\partial}$-lemma, however their Frölicher spectral sequence degenerates at first page.
\end{theorem}

\begin{proof}

We compute first the Hodge numbers of a diagonal Hopf manifold $X=W_{d_{\lambda}}$, where $\lambda \in (\mathbb{D}^*)^n$. We will take a spectral sequence approach, since this will simplify some computations from \cite{m}. By \cite[Subsection 5.2]{gro}, the Dolbeault cohomology of $X$ can be computed as $\mathbb{Z}$-equivariant cohomology of $W$. More precisely, for each fixed $k\in\N$, %the sheaf of germs of holomorphic $k$-forms $\Omega^k_{X}$, 
one can consider a spectral sequence $E_r^{\bullet, \bullet}$ whose second page is given by $$E_2^{p, q}=H^p(\mathbb{Z}, H^q(W, \Omega^k_{W})) \Rightarrow H^{p+q}(X, \Omega^k_{X}),$$
where $\mathbb{Z}=\langle d_\lambda \rangle$ acts on a class $[\omega]\in H^q(W, \Omega^k_{W})$ by $d_\lambda \cdot [\omega]=[d_{\lambda}^*\omega].$ 

Since $\mathbb{Z}$ is of rank 1, $E_2^{p, \bullet}=0$ for $p \geq 2$, therefore %$E_2^{\bullet, \bullet}$ consists only of 2 columns, meaning it will converge at second page and 
$E_{\infty}^{\bullet, \bullet}=E_2^{\bullet, \bullet}$ and we get
\begin{equation*}
    h^{k, l}_{\overline{\partial}}:=\mathrm{dim}_{\mathbb{C}}H^l(X, \Omega_X^k)=\mathrm{dim}_{\mathbb{C}}H^0(\mathbb{Z}, H^l(W, \Omega_W^k))+\mathrm{dim}_{\mathbb{C}}H^1(\mathbb{Z}, H^{l-1}(W, \Omega_W^k)).
\end{equation*} 
Furthermore, since $\Omega^k_W$ is free over $\mathcal{O}_W$ and since by \cite[Lemma 3.2]{h} we have that $H^p(W,\mathcal{O}_W)=0$ for $p\neq 0, n-1$, we readily infer that $h^{k,l}_{\overline{\partial}}=0$ for $1<l<n-1$ and we are left with computing
$$h^{k, 0}_{\overline{\partial}}= \mathrm{dim}_{\mathbb{C}}H^0(\mathbb{Z}, H^0(W, \Omega_W^k))=\mathrm{dim}_{\mathbb{C}} \, \mathrm{Ker} (\mathrm{Id}-d^*_{\lambda})_{|H^{0}(W, \Omega^{k}_W)}$$ 
$$h^{k, 1}_{\overline{\partial}}= \mathrm{dim}_{\mathbb{C}} H^1(\mathbb{Z}, H^0(W, \Omega_W^k)) =\mathrm{dim}_{\mathbb{C}}\mathrm{CoKer} (\mathrm{Id}-d^*_{\lambda})_{|H^{0}(W, \Omega^{k}_W)}$$
$$h^{k, n-1}_{\overline{\partial}}=\mathrm{dim}_{\mathbb{C}}H^0(\mathbb{Z}, H^{n-1}(W, \Omega_W^k))=\mathrm{dim}_{\mathbb{C}} \mathrm{Ker} (\mathrm{Id}-d^*_{\lambda})_{|H^{n-1}(W, \Omega^{k}_W)}$$
$$h^{k, n}_{\overline{\partial}}=\mathrm{dim}_{\mathbb{C}}H^1(\mathbb{Z}, H^{n-1}(W, \Omega_W^k))=\mathrm{dim}_{\mathbb{C}} \mathrm{CoKer} (\mathrm{Id}-d^*_{\lambda})_{|H^{n-1}(W, \Omega^{k}_W)}.$$

Clearly $h^{0,0}_{\overline{\partial}}=1$. We start by showing that $h^{k, 0}_{\overline{\partial}}=0$ for all  $k \neq 0$. Indeed, take a $d_\lambda$-invariant holomorphic $k$-form on $W$
\begin{equation}\label{omega}
\omega=\sum_{I} \sum_{\alpha \in \mathbb{N}^n} a^I_{\alpha} z^{\alpha}dz_{I},
\end{equation}
where $I=\{(i_1, \ldots, i_k) | 1\leq i_1 < \ldots < i_k \leq n\}$, $dz_{I}=dz_{i_1} \wedge \ldots \wedge dz_{i_k}$, $z^{\alpha}=z_1^{\alpha_1}\ldots z_n^{\alpha_n}$ and $a_{\alpha} \in \mathbb{C}$. The condition $d_{\lambda}^*\omega=\omega$ reads as 
\begin{equation}\label{inv}
    a_{\alpha}^{I}\lambda^{\alpha}\cdot \lambda_{i_1} \ldots \lambda_{i_k}=a_{\alpha}^{I}
\end{equation}
for any $I$, where $\lambda^{\alpha}=\lambda_{1}^{\alpha_1} \ldots \lambda_{n}^{\alpha_n}$. If $k \neq 0$, since $|\lambda_{i_1}|\leq ...\leq |\lambda_{i_p}|<1$, \eqref{inv} can take place if and only if $a_{\alpha}^I=0$, implying $\omega=0$. %Otherwise, clearly $H^0(X, \mathcal{O}_{X})=\mathbb{C}$.

%For computing $H^1(X, \mathcal{O}_{X})$, 
In order to show $h^{0,1}_{\overline{\partial}}=1$, we notice that 
$$\mathrm{CoKer} (\mathrm{Id}-d^*_{\lambda})_{|H^{0}(W, \mathcal{O}_W)}=\frac{H^0(W,  \mathcal{O}_{W})}{\mathrm{Im}(\mathrm{Id}-d^*_{\lambda})} \simeq \mathbb{C},$$ 
since any holomorphic function $f$ on $W$ of the form
$$f=\sum_{\alpha \in \mathbb{N}^n \setminus \{0\}} f_{\alpha}z^{\alpha}=g-d_\lambda^*g  \text{ for } g:=\sum_{\alpha \in \mathbb{N}^n \setminus \{0\}} \frac{f_{\alpha}}{1-\lambda^{\alpha}}z^{\alpha}$$
and the series defining $g$ converges, since $(1/(1-\lambda^\alpha))_{\alpha\in\N^n\setminus\{0\}}$ is bounded.

For $k\neq 1$, we have  $\frac{H^0(W,  \Omega^k_{W})}{\mathrm{Im}(\mathrm{Id}-d^*_{\lambda})}=0$ since for any $\omega$ presented as in \eqref{omega}, we have 
\begin{equation*}
    \omega=(\mathrm{Id}-d^*_{\lambda})(h),\quad h:=\sum_{I} \sum_{\alpha \in \mathbb{N}^n} \frac{a^I_{\alpha}}{1-\lambda^{\alpha}\cdot \lambda_{i_1} \ldots \lambda_{i_p}} z^{\alpha}dz_{I}\in H^0(W,\Omega^k_W).
\end{equation*}
This implies that $h^{k,1}_{\overline{\partial}}=0$ for $k\neq 1$.

By Serre duality, the only Dolbeault numbers left to compute are $h^{0,n-1}_{\overline{\partial}}$, $h^{1,n-1}_{\overline{\partial}}$, $h^{0,n}_{\overline{\partial}}$ and $h^{1,n}_{\overline{\partial}}$. 
By \cite[Lemma 3.2]{h}, we have
$$H^{n-1}(W,\mathcal{O}_W)\cong\{\sum_{\alpha\in(\Z_{<0})^n}a_\alpha z^\alpha \text{ convergent on }(\C^*)^n\} $$
and the induced $\Z$-action on the right hand-side is generated by the natural action of $d_\lambda$ on holomorphic functions on $(\C^*)^n.$ Then the same reasoning as above gives us the remaining Dolbeault numbers.

For the general case, we use that any primary Hopf manifold $W_\gamma$ with $\gamma\in G_\lambda$ is arbitrarily close to the diagonal one $W_{d_\lambda}$ via \ref{deformation}. Then, by the upper semi-continuity of the Hodge numbers (see \cite[Ch. 3, Thm. 4.12]{bs}), we get $h_{\overline{\partial}}^{p, q}(W_{\gamma}) \leq h_{\overline{\partial}}^{p, q}(W_{d_{\lambda}})$ for any degree $(p,q)$. In particular, the numbers $h^{p, q}_{\overline{\partial}}(W_{\gamma})$ are lower than $1$ for $(p, q) \in \{(0, 0), (0,1), (n-1,n), (n,n)\}$ and vanish in any other degree. The final result is obtained by applying the Fr\" olicher inequality $h^{0, 1}_{\overline{\partial}} \geq b_1(W_\gamma)=1$.
\end{proof}

For the computation of the Bott-Chern numbers, we will in fact use the existence of Vaisman metrics, cf. Section \ref{sec: Metrics}.

\begin{theorem} The Bott-Chern numbers of a primary Hopf manifold are given by
      \begin{equation*}
         h_{BC}^{p,q}=\begin{cases}
      1, & \text{if}\,\, (p, q) \in \{(0, 0), (1,1), (n-1,n), (n, n-1), (n,n)\}\\
      0, & \text{otherwise.}
    \end{cases}      
      \end{equation*} 
\end{theorem}   

\begin{proof}
A diagonal Hopf manifold admits a Vaisman metric, following the discussion of Section \ref{sec: Metrics}. For such manifolds, the Bott-Chern numbers  can be recovered from the Dolbeault cohomology, as shown in \cite[Thm 3.1, Cor. 5.2]{io}, retrieving the numbers in the statement. As a general primary Hopf manifold is arbitrarily close to a diagonal one, and since  \cite[Cor. 5.4]{io} states  that the Bott-Chern numbers remain constant for a small deformation of a compact Vaisman manifold, we infer the statement for any primary Hopf manifold. \end{proof}

     \begin{proposition} The Kodaira dimension of an arbitrary Hopf manifold is $-\infty$.
\end{proposition}

\begin{proof} 
Since any Hopf manifold is finitely covered by a primary one, it is enough to show the result for any Hopf manifold of the form $W_{\gamma}$, with $\gamma\in G_\lambda$ in normal form. We first show that for the diagonal Hopf $W_{d_{\lambda}}$, all plurigenera vanish. Indeed, the meromorphic form on $W$
$$\frac{dz_1 \wedge \ldots \wedge dz_n}{z_1 z_2 ... z_n}$$ is invariant under $d_\lambda$, hence descends to a form $\omega$ on $W_{d_\lambda}$. The canonical divisor $K_{W_{d_\lambda}}$ is then given by $div(\omega)=-W_1 - \ldots - W_n$, where $W_i$ is the complex hypersurface defined by $\{z_i=0\}$. % and is isomorphic to the $(n-1)$-dimensional Hopf manifold $W_{d_{\lambda_1, ..., \hat{\lambda_i}, ..., \lambda_n}}$. 
 Consequently, the canonical bundle $K_{W_{d_{\lambda}}}$ is anti-effective and so the plurigenera vanish for $W_{d_{\lambda}}$.

Consider the family of deformation $p:\mathcal{X} \rightarrow \mathbb{C}$ in \ref{deformation} such that $p^{-1}(0) \simeq W_{d_{\lambda}}$ and $p^{-1}(t) \simeq W_{\gamma}$, for any $t \neq 0$. By \cite[Ch. 3, Thm. 4.12]{bs}, for any $l>0$,  the plurigenera $P_t: =\mathrm{dim}_{\mathbb{C}}H^0(p^{-1}(t), K_{p^{-1}(t)}^l)$  give upper semi-continuous functions $t \mapsto P_t$, therefore the plurigenera of $W_\gamma$ vanish as well and $\mathrm{Kod}(W_{\gamma})=-\infty$. 
\end{proof}

The following result computes the algebraic dimension of a diagonal Hopf manifold in terms of the relations between the $\lambda_j's$. We note in particular that all possible values strictly lower than the complex dimension can be attained as algebraic dimensions. Of course there is no Moishezon Hopf manifold, as Hopf manifolds do not  satisfy the $\partial\overline{\partial}$-lemma. We also note that since the algebraic dimension is not upper-semicontinuous in general \cite{FP}, we cannot directly infer anything about the algebraic dimension of a non-diagonal Hopf manifold.

For any $\lambda\in(\C^*)^n$, let us introduce the group homomorphism 
$$\psi_\lambda \colon\Z^n\rightarrow \C^*, \quad m\mapsto \lambda^m.  $$

\begin{proposition}\label{thm: AlgDim}
Let $W_{d_\lambda}$ be an $n$-dimensional diagonal Hopf manifold. Then $a(W_{d_\lambda})=\mathrm{rank}(\ker \psi_\lambda).$ 
\end{proposition}
\begin{proof}
Let us first assume that $\ker\psi_\lambda=\{0\}$ and show that $W_{d_{\lambda}}$ has no meromorphic function, whence vanishing algebraic dimension. Assume, by contradiction, that $f$ is the pullback on $W$ of a non-constant meromorphic function on $W_{d_{\lambda}}$. Since $n\geq 2$, $f$ can be extended as a meromorphic function on $\mathbb{C}^n$ and be written as $\frac{g}{h}$, where $g$ and $h$ are holomorphic functions on $\mathbb{C}^n$ and $d_{\lambda}^*f=f$. Both $g$ and $h$ can be expressed as power series
\begin{equation*}
   g(z)=\sum_{m \in \mathbb{N}^n} a_mz^m, \quad 
    h(z)=\sum_{l \in \mathbb{N}^n} b_lz^l
\end{equation*}
satisfying $g(\lambda_1z_1, ..., \lambda_nz_n) h(z)=h(\lambda_1z_1, ..., \lambda_nz_n) g(z)$. This implies
\begin{equation*}
    \sum_{m \in \mathbb{N}^n}a_m\lambda^mz^m\sum_{l \in \mathbb{N}^n} b_lz^l=\sum_{m \in \mathbb{N}^n}a_mz^m\sum_{l \in \mathbb{N}^n} b_l\lambda^lz^l.
\end{equation*}
Take now the lowest degrees $k_1$ and $k_2$ in the expansions of $g$ and $h$, respectively and take $m_0=(m_{01}, \ldots, m_{0n})$ and $l_0=(l_{01}, \ldots, l_{0n})$ such that $m_{01}$ and $l_{01}$ are the smallest among the multi-indices $m$ and $l$ of degree $k_1$ and $k_2$, respectively.  Then 
\begin{equation*}
a_{m_0}b_{l_0}\lambda^{m_0}z^{m_0+l_0}=a_{m_0}b_{l_0}\lambda^{l_0}z^{m_0+l_0},
\end{equation*}
which gives $\lambda^{m_0}=\lambda^{l_0}$, contradicting the assumption. Thus indeed $a(W_{d_\lambda})=0.$

    In the general case, assuming that $k:=\mathrm{rank}(\ker\psi_\lambda)$, any $m\in\ker\psi_\lambda$ gives rise to the $d_\lambda$-invariant meromorphic function $z^m$ on $W$, and so to a meromorphic function $f_m$ on $W_{d_\lambda}$. In particular, we have $\C(f_m|m\in\ker\psi_\lambda)\subset \C(W_{d_\lambda})$. Now note that $\ker\psi_\lambda$ is a subgroup of $\Z^n$, therefore torsion free and so there exist $m_1,\ldots, m_k\in\ker\psi_\lambda$ so that $\ker\psi_\lambda=\Z\langle m_1,\ldots, m_k\rangle\cong\Z^k.$ We claim that the associated $k$ meromorphic functions $f_i=f_{m_i}$, $i\in\{1,\ldots, k\}$, are algebraically independent. Indeed, if there is some polynomial $0\neq P=\sum a_lX^l\in\C[X_1,\ldots, X_k]$ with $P(f_1,\ldots, f_k)=0$ then, for each $l\in\N^k$ with $a_l\neq 0$ we must have $\sum_{i=1}^kl_im_i=0$, which then implies that $l=0$ which is a contradiction. Therefore, we have $k\leq a(W_{d_\lambda})$. If $k=n-1$, this automatically implies that $a(W_{d_{\lambda}})=n-1$ since $a(W_{d_\lambda})<n.$

    For the case $k< n-1$, we still need  to show the converse inequality. In order to do so, note that there exists a subset $J\subset\{1,\ldots, n\}$ of cardinal $n-k$ so that $\ker\psi_\lambda\cap\langle e_j|j\in J\rangle=\{0\}$. After eventually reordering coordinates, we can assume that $J=\{k+1,\ldots, n\}$. By defining $\mu:=(\lambda_{k+1},..., \lambda_{n})\in(\mathbb{D}^*)^{n-k}$, it follows that $\ker\psi_\mu=\{0\}.$ At the same time, $Y:=\{z_1=...=z_k=0\}\subset W_{d_\lambda}$ is a complex submanifold of $W_{d_\lambda}$ of codimension $k$ which is biholomorphic to $W_{d_\mu}$. Since $\dim_\C Y>1$, the above discussion implies that $Y$ has algebraic dimension 0. Applying \cite[Thm 3.8 (ii)]{u}, we get $a(W_{d_{\lambda}})\leq a(Y)+\mathrm{codim}(Y)=k$.  We conclude therefore that $a(W_{d_\lambda})=k.$

\end{proof}

\begin{remark}
Note that for a generic $\lambda\in(\mathbb{D}^*)^n,$ $\psi_\lambda$ is injective and therefore $a(W_{d_\lambda})=0.$
\end{remark}

\begin{remark}
The simplest Hopf manifolds of positive algebraic dimension $k$ correspond to  $\lambda\in(\mathbb{D}^*)^n$ with $\lambda_1=...=\lambda_{k+1}$ and so that $\psi_\mu$ is injective, where $\mu=(\lambda_{k+1},..., \lambda_{n})\in(\mathbb{D}^*)^{n-k}$. In this case, the algebraic reduction is given by $\CP^k$. Indeed, consider the meromorphic map induced by the linear projection to the first $k+1$-coordinates:
\[ r\colon W_{d_\lambda}\dashrightarrow \mathbb{P}^{k}, \quad [z_1,\ldots, z_n]\mapsto (z_1:\cdots: z_{k+1}),\]
where $[z]$ denotes the image of $z\in W$ in $W_{d_\lambda}.$ Then $Z:=\{z_1=\ldots=z_{k+1}=0\}\subset W_{d_\lambda}$ is the indeterminacy locus of $r$ and $r$ lifts to a holomorphic map $r^*$ on the blow-up of $W_{d_\lambda}$ along $Z$  
\[ \mathrm{Bl}_ZW_{d_\lambda}=\{([z],(u_1:\cdots u_{k+1}))\mid z_iu_j=z_ju_i, \forall i,j\in\{1,\ldots, k+1\}\}\subset W_{d_\lambda}\times \mathbb{P}^{k}\]
which coincides with the projection to the second factor. In particular, the fibers of $r^*$ are connected, all isomorphic to the  $n-k$-dimensional diagonal Hopf manifold $W_{d_\mu}$. Since $a(W_{d_\lambda})=k$ by the above, it follows that $(r^*,\mathbb{P}^{k})$ is the algebraic reduction of $W_{d_\lambda}.$
\end{remark}

     \section{Hermitian metrics}\label{sec: Metrics}

Due to their topological restrictions, Hopf manifolds cannot support a K\" ahler metric. Therefore, a relevant question for Hermitian non-K\" ahler geometry is what kind of special metrics Hopf manifolds possess. We recall briefly some of the definitions of these special geometries and we shall focus more on the one which has been greatly linked to their study, which is {\em locally conformally K\" ahler (lcK) geometry}. 

\subsection{lcK and Vaisman metrics}

A Hermitian metric $\omega$ on a complex manifold $X$ is called {\em locally conformally K\" ahler} (lcK) if $d\omega=\theta \wedge \omega$ for a closed one-form $\theta$, called the Lee form of the metric. If moreover $\theta$ is parallel with respect to the Levi-Civita connection of the metric, $\omega$ is called {\em Vaisman}.

The existence of lcK metrics on Hopf manifolds has a long history, going back to Vaisman \cite{V}, who proved than on a diagonal Hopf manifold $W_{d_{\lambda}}$ with $|\lambda_1|= \ldots = |\lambda_n|$, the Hermitian metric induced by 
$$\omega=\frac{i\del\overline{\del}||z||^2}{||z||^2}$$
is Vaisman. In fact, this is the only case where an explicit lcK metric can be written down on a Hopf manifold. Next, Gauduchon and Ornea \cite{go} proved that any Hopf surface admits lcK metrics. On the general diagonal Hopf surface, they construct an lcK metric of the form $\omega=f^{-1}i\del\overline{\del}f$, where the potential $f$ is given as the implicit unique solution of some functional equation, and then by direct computation they infer that this metric is Vaisman. On non-diagonal Hopf surfaces, the existence of the metric is inferred via a deformation argument. A different deformation argument for constructing lcK metrics on Hopf surfaces was used by Belgun \cite{bel}, who also showed that non-diagonal Hopf surfaces are lcK, but never Vaisman.

Since then, many different arguments for the existence of lcK metrics on primary Hopf manifolds of any dimension have been given. For instance the direct generalisation of \cite{go} to any diagonal Hopf manifolds appears in  \cite{ko} (see also \cite[Sect. 2.6.1]{nico}), where the authors also show that the constructed metric is Vaisman. One can then use a deformation argument of \cite{ov10} and the isotrivial family of \ref{deformation} in order to deduce the existence of lcK metrics on any non-diagonal Hopf manifold. Other works of Ornea-Verbitsky also construct  such metrics using more involved arguments (see \cite{ov16} for linear Hopf manifolds and \cite{ov23} for non-linear Hopf manifolds). 

A unifying point of view for the existence of lcK metrics on any primary Hopf manifold, diagonal or not, appears in an argument of  Brunella \cite{bru} formulated for Kato surfaces. This construction uses the fact that a conformal class of an lcK metric on $X$ is equivalent to a conformal class of a K\"ahler metric on the universal cover $\tilde X$ on which the deck group acts by homotheties. Without going into the definition of a Kato manifold, let us make explicit the argument from \cite{bru} for Hopf manifolds:

\begin{theorem}(Brunella)\label{Brunella}
    Every primary Hopf manifold carries lcK metrics.
\end{theorem}

\begin{proof} 
Let $W_\gamma$ be the $n$-dimensional Hopf manifold in question, and for any $r>0$, let us denote by $\mathbb{B}_r=\{z\in\C^n,||z||^2<r\}.$ Since $\gamma$ is a contraction, we can pick $0<s<1$ so that $\gamma(\overline{\mathbb{B}}_1)\subset\mathbb{B}_s$.

Let $f_0(z):=||z||^2$, $f(z):=f_0\circ\gamma^{-1}$. Note that as $f$ is strictly plurisubharmonic, via the maximum principle we have that  $m:=\min_{\partial\mathbb{B}_s}f<M:=\max_{\partial\mathbb{B}_1}f$, thus for any $A>\frac{M-m}{1-s^2}>0$ there exists $c\in (M-A,m-As^2)$ and we have $A\cdot f_0+c>f$ on $\partial\mathbb{B}_1$ and $A\cdot f_0+c<f$ on $\partial\mathbb{B}_s$. Thus the function $\psi$ defined as the regularised maximum between $A\cdot f_0+c$ and $f$ on $\mathbb{B}_1-\mathbb{B}_s$ and equal to $f$ on $\mathbb{B}_s$ is smooth and defines a K\"ahler metric $\omega:=i\partial\overline{\partial}\psi$. Moreover, since $\gamma^*\psi|_{\gamma(\partial\mathbb{B}_1)}=f_0|_{\partial\mathbb{B}_1}$, we have that
$$\gamma^*\omega|_{\gamma(\partial\mathbb{B}_1)}=\frac{1}{A}\omega|_{\partial\mathbb{B}_1}.$$
Since $\mathbb{B}_1-\gamma(\mathbb{B}_1)\subset W$ is a fundamental domain for the action of $\Gamma:=\langle\gamma\rangle$ on $W$, it follows that the metric $\omega$ corresponds to the conformal class of an lcK metric on $W_\gamma$. \end{proof}

In order to detect whether a given lcK manifold can be Vaisman or not, a useful tool is the so called Lee group. For a Vaisman metric $(\omega,\theta)$  on $X$, the smooth vector field which is metric dual to $\theta$, $B:=\theta^{\sharp_g}\in\mathcal{C}^\infty(TX)$, together with $JB$, generate a complex one-dimensional Lie group acting by biholomorphic isometries 
$$G=\exp\langle B, JB\rangle\subset\mathrm{Aut}_0(X,\omega)$$ called \textit{the Lee group} of $X$. If $X$ is compact, then $\mathrm{Aut}_0(X,\omega)$ is a compact real Lie group, so that we can also define the compact Lie group $H:=\overline{G}\subset\Aut_0(X,\Omega)$ as the closure of $G$. As $G$ is abelian, $H$ is a real compact torus of dimension at least $2$. Let us denote by $\Lh$ the Lie algebra of $H$. Note that $\Lg=\langle B,JB\rangle\subset\Lh\cap J\Lh\neq \{0\}$ and in fact we have equality here, by \cite[Prop. 3]{is}. More specifically, a compact lcK manifold bears a Vaisman metric if and only if it bears a complex Lie group $G\subset\Aut_0(X)$ which is contained in a real compact Lie subgroup $H\subset\Aut_0(X)$. Moreover, if it exists, such a group is unique. 

This, together with \ref{prop: HLee},  implies: 

\begin{proposition}
    A primary Hopf manifold $W_\gamma$ bears a Vaisman metric if and only if $\gamma$ is equivalent to a diagonal automorphism $d_\lambda$. Moreover, in this case, the Lee vector field is a positive multiple of 
    $$B=\Re \left(\sum_{j=1}^n\ln |\lambda_j|z_j\frac{\del}{\del z_j}\right).$$
\end{proposition}
\begin{proof}
The first part is contained in \ref{prop: HLee}. In particular, in the diagonal case it was shown that $\Lg=\Lh\cap J\Lh=\langle Z\rangle_\C$, where
\[Z=\sum_{j=1}^n\ln |\lambda_j|z_j\frac{\del}{\del z_j}.\]
Since a Lee vector field $B$ must satisfy $\theta(B)>0$ and $\theta(JB)=0$ for a one-form $\theta$ cohomologuous to $-d\ln|z|^2$, the only possibility for $B$ is the one given.

\end{proof}

\begin{remark}
    The above results only discuss the existence of lcK metrics on primary Hopf manifolds. However, the deformation result of \cite{ov10} together with \ref{deformation} ensure the existence of \textit{lcK metrics with potential} on these manifolds, i.e. metrics of the form $\omega=f^{-1}i\partial\overline{\partial}f$ on $\tilde X=W$ so that the deck group acts by homotheties on $f$. Since for a general Hopf manifold $X$ one has, by \cite[Thm. 2.3]{ha}, that $\pi_1(X)\cong\Gamma\ltimes F$, with $\Gamma<\Aut(W)$ infinite cyclic and $F\triangleleft \pi_1(X)$ finite, %and $\Gamma\cong\Z$,and $F<\Aut(W)$ a finite normal, $ 
   one can first construct an lcK metric with potential on $W/\Gamma$ and then average the smooth, positive, strictly plurisubharmonic function $f$ to an $F$-invariant one of the same kind $h$ on $W$. This then allows one to descend the lcK metric $h^{-1}i\partial\overline{\partial}h$  from $W$ directly to $X$. 
\end{remark}

\subsection{Non-existence of other distinguished metrics}

If a Hermitian metric $\omega$ satisfies $\partial\overline{\partial}\omega=0$, then $\omega$ is called {\em pluriclosed}. This condition coincides with the notion of {\em Gauduchon} metric in complex dimension 2, therefore such metrics always exist on surfaces by \cite{Gaud}. A Hermitian metric $\omega$ on an $n$-dimensional complex manifold is called {\em balanced} if $d\omega^{n-1}=0$. More generally, $\omega$ is called {\em strongly Gauduchon} if $\partial \omega^{n-1}$ is $\overline{\partial}$-exact. 
 
The problem of existence of several special metrics can be solved easily with a negative answer if the Hopf manifold is Vaisman, as  shown in \cite{ao}. Even though most of the arguments cannot be transfered in general from the Vaisman case to any Hopf manifold, there are still neat results about Hopf manifolds in general. The following two results rule out the existence of pluriclosed and strongly Gauduchon metrics. 

\begin{theorem} Hopf manifolds of complex dimension larger than $2$ do not admit pluriclosed metrics.
\end{theorem}

\begin{proof} The proof rests purely on cohomological properties. %and therefore, does not depend on whether the Hopf manifold is Vaisman or not. 
Moreover, it is clearly enough to establish the result on primary Hopf manifolds.  We reproduce briefly the argument in \cite[Thm. 5.16]{c}. By \ref{dolbeault}, any primary Hopf manifold $X$ satisfies $H^{2, 1}_{\overline{\partial}}(X)=H^{3, 0}_{\overline{\partial}}(X)=0$. If a pluriclosed metric $\omega$ existed, then $\partial \omega=\overline{\partial}\sigma$ for some $\sigma \in \Omega^{2, 0}(X)$ with $\partial \sigma \in H^{3,0}_{\overline{\partial}}(X)=0$. This amounts to $\Omega:=\omega - \sigma - \overline{\sigma}$ being a closed 2-form. Moreover,  as $\Omega(V, JV)>0$ for any vector field $V$, $\Omega$ is non-degenerate, hence a symplectic form. However this cannot exist on $X$, as $H^2(X)=0$.
\end{proof}

\begin{theorem} Hopf manifolds do not admit strongly Gauduchon metrics. In particular, they cannot be balanced. 
\end{theorem}

\begin{proof} 
Again it is enough to show this for primary Hopf manifolds.
The argument is given in \cite[Section 2]{pop14}, by using the characterization of existence of a strongly Gauduchon metric in terms of currents. As proven in \cite[Prop. 3.3]{pop09}, a complex manifold carries a strongly Gauduchon metric if and only if there exists no non-zero current $T$ of type $(1, 1)$ such that $T \geq 0$ and $T$ is $d$-exact. However, such a current exists on a primary Hopf manifold $X=W_\gamma$. Indeed, if $\gamma$ is in normal form and $n=\dim_\C X$, then $\{z_n=0\}$ defines a hypersurface in $W_\gamma$,  therefore its current of integration gives precisely a current obstructing the existence of strongly Gauduchon metrics, as $H^2(X)=0$. 
\end{proof}

\end{document}